\documentclass[12pt,a4paper]{amsart}
\usepackage{amsmath, amssymb, amsfonts,enumerate}
\usepackage[all]{xy}
\usepackage{amscd}

\newtheorem{theorem}{Theorem}[section]
\newtheorem{lemma}[theorem]{Lemma}

\newtheorem{proposition}[theorem]{Proposition}

 \theoremstyle{definition}
 \newtheorem{definition}[theorem]{Definition}

\numberwithin{equation}{section}
\newcommand {\Z}{\mathbb{Z}} 

\newcommand{\PP}{\mathcal{P}}

\newcommand{\UU}{\mathcal{U}}

\DeclareMathOperator{\Stab}{Stab}

\begin{document}
\title{Sensitivity and Devaney's chaos in uniform spaces}
\author{Tullio Ceccherini-Silberstein}
\address{Dipartimento di Ingegneria, Universit\`a del Sannio, C.so
Garibaldi 107, 82100 Benevento, Italy}
\email{tceccher@mat.uniroma3.it}
\author{Michel Coornaert}
\address{Institut de Recherche Math\'ematique Avanc\'ee,
UMR 7501,                                             Universit\'e  de Strasbourg et CNRS,
                                                 7 rue Ren\'e-Descartes,
                                               67000 Strasbourg, France}
\email{coornaert@math.unistra.fr}
\begin{abstract}
We give sufficient conditions for sensitivity of continuous group actions on uniform spaces. 
\end{abstract}
\maketitle

\section{Introduction}

Sensitivity is an important property of chaotic dynamical systems
which is related to what  was  popularized  as the ``butterfly effect" in the 1970s.
When the phase space is equipped with a metric, the definition for group actions goes as follows.
One says that an action of a group $G$
on a metric space $(X,d)$    
is  \emph{sensitive} if there exists a constant $\varepsilon > 0$ such that, 
 for every point $x \in X$ and every neighborhood $N$ of $x$, there exist a point $y \in N$    and an element $g \in G$ such that $d(gx,gy) \geq \varepsilon$ (cf.  \cite{banks}, 
\cite{devaney}, \cite{glasner-weiss-sensitive}, \cite{katok-hasselblatt-first-course},  
\cite{kontorovich}, \cite{silverman}).
 
The goal of this note is to give a definition of sensitivity in the case when the phase space $X$ is only equipped with an uniform structure instead of a metric
and to extend to this more general setting two classical results providing sufficient conditions for sensitivity.  It turns out that this extension is quite natural and we think it may help to a better
understanding of the roles of the various topological and dynamical concepts involved.
\par
The theory of uniform spaces was developed in \cite{weil-uniforme}. A uniform space is a set $X$  equipped with a uniform structure, i.e., a set $\UU$ consisting of subsets of the Cartesian square $X \times X$. These subsets are called the \emph{entourages} of the uniform space and are  required to verify a certain list of axioms (see Section \ref{s:background}).  Heuristically, two points $x, y \in X$ are ``close" when the pair $(x,y)$ belongs to a ``small" entourage.
Sensitivity for group actions on uniform spaces may be defined as follows.  

 \begin{definition}
 One says that the action of a group $G$ on a uniform space $(X,\UU)$  has \emph{sensitive dependence on initial conditions}, or more briefly  that the action  is \emph{sensitive}, 
 if there is an entourage $U \in \UU$ such that, for all $x \in X$
and every neighborhood $N$ of $x$, there exist a point $y \in N$ and an element $g \in G$ such that $(gx,gy) \notin U$.
Such an entourage $U$ is then called a \emph{sensitivity entourage} for the dynamical system 
$(X,G)$.
 \end{definition}

When the uniform structure comes from a metric, this definition is equivalent to the one given above.
\par
Note that every uniform space admitting a sensitive action must be \emph{perfect}, i.e., without isolated points.
In particular, every Hausdorff uniform space admitting a sensitive action is infinite.
\par
Sensitivity is a weak local version of expansivity.
We recall  that  the action of a group $G$ on a uniform space $(X,\UU)$ 
is called  \emph{expansive} if there is a an entourage $U \in \UU$ such that, for all distinct points 
$x, y \in X$,
  there exists an element $g \in G$ such that $(gx,gy) \notin U$.
Such an entourage $U$ is then called an \emph{expansitivity entourage} for the dynamical system $(X,G)$.
 It is clear that if the action of a group $G$ on a perfect uniform space is expansive then this action is also sensitive (any expansitivity entourage is a sensitivity entourage).
Note that if a group $G$ acts expansively on a uniform space $X$ 
and $Y \subset X$ is a $G$-invariant subset, then the action of $G$ on $Y$ is also expansive 
so that it is sensitive if and only if $Y$ is perfect. 
\par
Recall the following standard definitions.
Let $X$ be a set equipped with an action of a group $G$.
The orbit of a point $x \in X$ is the subset $Gx = \{gx : g \in G\} \subset X$ and its \emph{stabilizer}
is the subgroup  
$\Stab_G(x) = \{ g \in G : gx = x\} \subset G$.
A point $x \in X$ is called \emph{periodic} if
its orbit is finite. Equivalently, $x$ is periodic if and only if its stabilizer is of finite index in $G$.
Suppose now that  $X$ is a topological space.
One says that the action of $G$ on $X$ is  \emph{continuous} if, for each $g \in G$, the map $x \mapsto gx$ is continuous on $X$.
The action of $G$ on $X$ is called
\emph{topologically transitive} 
 if, given any two non-empty open subsets $V$ and $W$ of $X$, there exists an element $g \in G$ such that $g V$ meets $W$. 
Our main result is the following.

\begin{theorem}
\label{t:main-theorem}
Let $X$ be an infinite Hausdorff uniform space equipped with a continuous and topologically transitive action of a group $G$. 
Suppose in addition that  $X$ admits a dense set of periodic points.
Then the action of $G$ on $X$ is sensitive.
 \end{theorem}

In the case when $X$ is a metric space,
the preceding result was previously obtained in \cite{banks} and \cite{silverman} for $G = \Z$, and in 
\cite{kontorovich} for an arbitrary group $G$.
\par
We say that the action of a group $G$ on a uniform space $X$ is \emph{chaotic} in the sense of Devaney if it is topologically transitive and periodic points are dense in $X$ (cf. \cite{devaney}).
With this definition,  Theorem \ref{t:main-theorem}
may be rephrased by saying that any continuous action of a group $G$ on an infinite Hausdorff uniform space $X$ which is chaotic in the sense of Devaney is also sensitive.
\par
An action of a group $G$ on a topological space $X$ is called \emph{topologically mixing} if, given any two non-empty open subsets $V$ and $W$ of $X$, there exists
a finite subset $K \subset G$ such that $gV$ meets $W$ for all $g \in G \setminus K$.
It trivially follows from this definition that every action of a finite group on a topological space is topologically mixing and that every topologically mixing action of an infinite group on a topological space is topologically transitive.
We shall also establish the following result which gives another sufficient condition for sensitivity of group actions on uniform spaces.

\begin{theorem}
\label{t:mixing-implies-sensitive}
Every  topologically mixing continuous action of an infinite group on a  
Hausdorff uniform space that is not reduced to a single point is sensitive.
\end{theorem}

For group actions on metric spaces,
the preceding result may be found in 
\cite[Proposition 7.2.14]{katok-hasselblatt-first-course}.

In Section \ref{examples} we present some examples of sensitive actions on non-metrizable uniform spaces.

\section{Background material on uniform spaces}
\label{s:background}

 Let $X$ be a set.
 \par 
 We denote by $\Delta_X$ the diagonal in $X \times X$, that is the set
$\Delta_X = \{ (x,x) : x \in X \}$.
\par
The \emph{inverse} $\overset{-1}{U}$ of a subset $U \subset X \times X$ is the subset of $X \times X$ defined by
$\overset{-1}{U} = \{ (x,y) : (y,x) \in U \}$.
One says that $U$ is \emph{symmetric} if $\overset{-1}{U} = U$.
Note that $U \cap \overset{-1}{U}$ is symmetric for any $U \subset X \times X$.
\par  
 We define the \emph{composite} $U \circ V$ of two subsets $U$ and $V$ of $X \times X$  by
  $$
U \circ V = \{ (x,y): \text{ there exists  } z \in X \text{  such that  } (x,z) \in U \text{  and  } (z,y) \in V \}  \subset X \times X.
$$

\begin{definition}
Let $X$ be a set. A \emph{uniform structure}\index{uniform ! --- structure} on $X$ is a non--empty set $\UU$ of subsets of $X \times X$ satisfying the following conditions:
\begin{enumerate}[(UN-1)]
\item if $U \in \UU$, then $\Delta_X \subset U$;
\item if $U \in \UU$ and $U \subset V \subset X \times X$, then $V \in \UU$;
\item if $U \in \UU$ and $V \in \UU$, then $U \cap V \in \UU$;
\item if $U \in \UU$, then $\overset{-1}{U} \in \UU$;
\item if $U \in \UU$, then there exists $V \in \UU$ such that 
$V \circ V \subset U$.
\end{enumerate}
The elements of $\UU$ are then called the \emph{entourages} of the uniform structure and the set $X$ is called a \emph{uniform space}.
\end{definition}

 Note that conditions (UN-3), (UN-4), and (UN-5) imply that, for any entourage $U$ there exists a symmetric entourage $V$ such that $V \circ V \subset U$.
\par
Let $X$ be a set and let   $U \subset X \times X$. 
Given a point $x \in X$, we define the subset $U[x] \subset X$ by
$U[x] = \{y \in X : (x,y) \in U \}$.
\par
If $X$ is a uniform space,
there is an induced topology on $X$ characterized by the fact that the neighborhoods of an arbitrary point $x \in X$ consist of the sets $U[x]$, where $U$ runs over all entourages of $X$.
This topology is Hausdorff if and only if the intersection of all the entourages of $X$ is reduced to the diagonal $\Delta_X $.
\par
 If $(X,d)$ is a metric space, there is a natural uniform structure on $X$ whose entourages are the sets $U \subset X \times X$ satisfying the following condition:
there exists a real number $\varepsilon >0$ such that
$U$ contains all pairs $(x,y) \in X \times X$ such that  $d(x,y) < \varepsilon$.
The topology associated with this uniform structure is then the same as the topology induced by the metric.

Uniform structures were introduced by Andr\'e Weil \cite{weil-uniforme}.
The reader is referred to \cite[Ch. 2]{bourbaki}, \cite[Ch. 6]{kelley}, and \cite{james}
for a detailed exposition of the general theory of uniform spaces. See also
\cite[Appendix B]{ca-and-groups-springer}.

\section{Proofs} 

\begin{lemma}
\label{l:finite-sets-bounded-away}
Let $A$ and $B$ be two disjoint finite subsets of a Hausdorff uniform space $X$.
Then there exists an  entourage $W$ of $X$ such that $A \times B$ does not meet $W$. 
 \end{lemma}

\begin{proof}
As $X$ is Hausdorff, we can find, for all $a \in A$ and $b \in B$,
an entourage $V_{a,b}$ of $X$ such that $(a,b) \notin V_{a,b}$.
Then the entourage
$$
W = \bigcap_{(a,b) \in A \times B} V_{a,b}
$$
 has the required property.
\end{proof}

\begin{proof}[Proof of Theorem \ref{t:main-theorem}]
We first claim that we can find an entourage $V$ of $X$ such that, for all $x \in X$, there
exists a finite orbit $C \subset X$ that does not meet $V[x]$.
   Indeed, the hypotheses that $X$ is Hausdorff, infinite, and contains a dense set of periodic points, imply that we can find two disjoint finite orbits $A$ and $B$ in $X$.
By Lemma \ref{l:finite-sets-bounded-away},
there is  an entourage $W$ of $X$ such that $A \times B$ does not meet $W$.
Now, if we take a symmetric    entourage $V$ satisfying $V \circ V \subset W$,
then there is no $x \in X$ such that $ A$ and $ B$ both meet $V[x]$.
 This proves the claim.
\par
Let $V$ be an entourage of $X$ satisfying the conditions of the preceding claim and
let $U$ be a symmetric  entourage of $X$ such that
\begin{equation}
\label{e:def-U}
U \circ U \circ U \circ U 
\subset V.
\end{equation}
Let us show that $U$ is a sensitivity entourage for the action of $G$ on $X $.
 Let $x \in X$ and let $N$ be a neighborhood of $x$.
\par
As periodic points are dense in $X$, we can find a periodic point $p \in X$ such that
\begin{equation}
\label{e:p-belongs-N-U-x}
p \in N \cap U[x].
\end{equation}
Denote by $H$ the stabilizer of $p$ in $G$    and let $T\subset G$
be a complete set of representatives for the left cosets of $H$ in $G$, so that we have 
$G = \coprod_{t \in T} tH$.
Note that $T$ is finite since $H$ is of finite index in $G$.
\par
As $V$ satisfies our first claim, 
we can find a finite orbit $C$ 
such that
\begin{equation}
\label{e:C-far-from-x}
C \cap V[x] = \varnothing.
\end{equation}
Choose an arbitrary point  $q \in C$. Then observe that the set
 $$
I = \bigcap_{t \in T}  t U[t^{-1}q]
$$
is a neighborhood of $q$ since it is a finite intersection of neighborhoods of $q$.
\par
As the action of $G$ on $X$ is topologically transitive, 
we can find a point
$z \in N \cap U[x]$ and an element $g_0 \in G$ such that $g_0z \in I$.
The element $g_0$ can be uniquely written in the form $g_0 = t_0h_0$, where $t_0 \in T$ is the representative of the class $g_0 H$ and $h_0 \in H$.
We have
\begin{equation}
\label{e:h0z-near}
h_0z = t_0^{-1}g_0z \in t_0^{-1}I \subset t_0^{-1}(t_0U[t_0^{-1}q]) = U[t_0^{-1}q],
\end{equation}
so that
\begin{equation}
(h_0z,t_0^{-1}q) \in U.
 \end{equation}
 We now claim  that we always have
 \begin{equation}
 \label{e:alternative}
(h_0x,p) \notin U \quad \text{or} \quad (h_0x,h_0z) \notin U.
 \end{equation}
 Indeed, suppose on the contrary that
$ (h_0x,p)$ and $(h_0x,h_0z)$ both belong to $U$.
Since $U$ is symmetric, this would imply $(p,h_0z) \in U \circ U$, and hence
$(x,t_0^{-1}q) \in U \circ U \circ U \circ U$ since $(x,p) \in U$ by \eqref{e:p-belongs-N-U-x} and 
$(h_0z,t_0^{-1}q) \in U$ by \eqref{e:h0z-near}.
This would contradict \eqref{e:C-far-from-x}
because $U \circ U \circ U \circ U  \subset V$
by \eqref{e:def-U} and $t_0^{-1}q \in C$.
\par
Observe now that $h_0p = p$ since $h_0 \in H$.
Thus, we deduce
from \eqref{e:alternative}  that we can always find  a point $y \in N$ and an element $g \in G$ such that $(gx,gy) \notin U$.
Indeed,
we can take $g = h_0$ and either $y = p$ or  $y = z$. 
\end{proof}

\begin{proof}[Proof of Theorem \ref{t:mixing-implies-sensitive}]
Let $X$ be a  Hausdorff uniform space equipped with a continuous and
topologically mixing action of an infinite group $G$.
Suppose that  $x_1$ and $x_2$ are two distinct points in $X$.
Since $X$ is Hausdorff we can find an entourage $V$ of $X$ such that
\begin{equation}
\label{x-1-x-2-not-in-V}
(x_1,x_2) \notin V.
\end{equation} 
Let $U$ be a symmetric entourage of $X$ such that
\begin{equation}
\label{u-4-in-v}
U \circ U \circ U \circ U \subset V.
\end{equation}
Let us show that $U$ is a sensitivity entourage for the action of $G$ on $X$.
Let $x \in X$ and let $N$ be a neighborhood of $X$.
As the action of $G$ on $X$ is topologically mixing, we can find, for $i=1,2$, 
a finite set $F_i \subset G$ such that we have
$$
g(N \cap U[x]) \cap U[x_i] \neq \varnothing
$$
for all $g \in G \setminus F_i$.
Since $G$ is infinite, the set $G \setminus (F_1 \cup F_2)$ is not empty.
Choose an element $g \in G \setminus (F_1 \cup F_2)$.
Then we can find
$y_1, y_2 \in N \cap U[x]$ such that
\begin{equation}
\label{gy-i-U-x-i}
gy_1 \in U[x_1] \quad \text{and} \quad gy_2 \in U[x_2].
\end{equation}
From  \eqref{gy-i-U-x-i}, \eqref{u-4-in-v}, and \eqref{x-1-x-2-not-in-V} we deduce that 
$$
(gy_1,gy_2) \notin U \circ U
$$
and therefore
\begin{equation}
\label{alternative}
(gx,gy_1) \notin U \quad \text{or} \quad (gx,gy_2) \notin U.
\end{equation}
It follows from \eqref{alternative} that we can always find a point $y \in N$ (namely, $y=y_1$ or $y=y_2$) and an element $g \in G$ such that $(gx,gy) \notin U$.
\par
This shows that $U$ is a sensitivity entourage for the action of $G$ on $X$.
\end{proof} 

\section{Examples from symbolic dynamics}
\label{examples}
Symbolic dynamics provides many interesting examples of sensitive actions on non-metrizable uniform spaces.
Indeed, let $A$ be a (possibly infinite) set having more than one element and let $G$ be an infinite group.
Consider the set $A^G$ consisting of all maps $x \colon G \to A$.
The \emph{shift} on $A^G$ is the action of $G$ on $A^G$  defined by
$gx(h) = x(g^{-1}h)$ for all $g,h \in G$ and $x \in A^G$.
We equip $A^G$ with its \emph{prodiscrete} uniform structure. This is the uniform structure admitting as a base of entourages the sets
$$
W(\Omega)= \{(x,y) \in A^G \times A^G : x\vert_\Omega = y\vert_\Omega\},
$$
where $\Omega$ runs over all finite subsets of $G$ (see \cite{ca-and-groups-springer}).
Then the shift action on $A^G$ is continuous and expansive (the set $W(\{1_G\})$ is an expansivity entourage). As $A^G$ is perfect, the shift on $A^G$ is sensitive.
Moreover, if $X \subset A^G$ is a $G$-invariant subset, then the restriction 
of the shift to $X$ is sensitive if and only if $X$ is perfect.

If $G$ is uncountable then $A^G$ is not metrizable (not even first countable).

Recall that a group is called \emph{residually finite} if the intersection of its subgroups of finite index is reduced to the identity element. In other words, residually finite groups are those groups whose elements can be distinguished after taking finite quotients.
The class of residually finite groups includes all virtually polycyclic groups 
(and in particular all finite groups, all finitely generated nilpotent groups and therefore all finitely generated abelian groups), all free groups, and all finitely generated linear groups; moreover, it is closed under the operations of taking: subgroups, direct products, and projective limits. In particular,
the following constitute examples of uncountable residually finite groups (see, e.g., \cite[Section 2]{ca-and-groups-springer}):

\begin{itemize}
\item{any infinite direct product of non-trivial residually finite 
(e.g. of non-trivial finite) groups;}
\item{the group $Aut(T_2)$ of automorphisms of the infinite binary rooted tree $T_2$;}
\item{the group $\Z_p$ of $p$-adic integers, for $p$ any prime number;}
\item{any uncountable free group.}
\end{itemize}

It is well known that periodic points are dense in $A^G$ if and only if the group $G$ is residually finite (see for example \cite[Theorem 2.7.1]{ca-and-groups-springer}). Collecting together these facts, we have:

\begin{proposition} 
Let $G$ be an infinite group and $A$ a (possibly infinite) set having more than one element. Then the shift on $A^G$ is chaotic in the sense of Devaney if and only if $G$ is residually finite. Moreover, if $G$ is uncountable then $A^G$ is not metrizable. \qed
\end{proposition}

Let $G$ be a group and $A$ a finite set. A closed $G$-invariant subset of $A^G$ is called a
\emph{subshift}. 

One says that a subshift $X \subset A^G$ is \emph{of finite type} if there exist a finite subset $\Omega \subset G$ and a subset $\PP \subset A^\Omega := \{p \colon \Omega \to A\}$ such that $X$ consists of all  $x \in A^G$ for which the restriction of the configuration $g x$ to 
$\Omega$ belongs to $\PP$ for all $g \in G$.

One says that a subshift $X \subset A^G$ is \emph{strongly irreducible} if there exists a finite subset  $\Delta \subset G$  such that,   
if $\Omega_1$ and $\Omega_2$ are finite subsets of $G$ which are sufficiently far apart in the sense that the sets $\Omega_1$ and $\Omega_2\Delta$ do not meet, 
then, given any two configurations $x_1$ and $x_2$ in $X$, there exists a configuration $x \in X$ which coincides with $x_1$ on $\Omega_1$ and with $x_2$ on $\Omega_2$.

One says that a subshift $X \subset A^G$ is \emph{topologically mixing} if the 
action of the $G$-shift on $X$ is topologically mixing.
This is equivalent to the following condition: for any finite subset $\Omega 
\subset G$ and any two configurations $x_1, x_2 \in X$, there exists a finite 
subset $F \subset G$ such that, for all $g \in G \setminus F$, there exists a 
configuration $x \in X$ which coincides with $x_1$ on $\Omega$ and with 
$x_2$ on $g\Omega$.

We remark that every strongly irreducible subshift is topologically mixing (see for example
\cite[Proposition 3.3]{myhill}) but there exist topologically mixing subshifts which are not 
strongly irreducible (e.g. the \emph{Ledrappier subshift}, see \cite[Remark 2.6]{dppsisft}).
However, it is known that every topologically mixing sofic subshift over $\Z$ is strongly 
irreducible (see for example \cite[Corollary 1.3]{myhill}).

We then have:

\begin{proposition} 
\label{propo2}
Let $G$ be an infinite residually finit group and $A$ a finite set. 
Suppose that $X \subset A^G$ is a strongly irreducible subshift of finite type
containing at least one periodic configuration.
Then, unless $X$ is reduced to a single configuration,  
the $G$-shift on $X$ is chaotic in the sense of Devaney.
\end{proposition}

\begin{proof}
Clearly every strongly irreducible subshift that is not reduced to a single 
configuration is perfect. As the shift action on $X$ is expansive, 
it follows that it is sensitive
(this may also be deduced from  Theorem \ref{t:mixing-implies-sensitive} 
since every strongly irreducible subshift is topologically mixing).
Density of periodic configurations in $X$ follows from \cite[Theorem 1.1]{dppsisft}. 
\end{proof}

\noindent
{\bf  Acknowledgments.} We express our deepest gratitude to the referee for her/his most
careful reading of our manuscript and for several useful comments and remarks.

\end{document}